\documentclass[12pt]{article}
\usepackage{graphicx}
\usepackage{amscd}
\usepackage{amsmath,amsthm}
\usepackage{caption}
\usepackage{amsfonts}
\usepackage{amssymb}
\usepackage{mathrsfs}
\usepackage{multicol}
\usepackage{color}
\usepackage{authblk}

\newcommand{\be}{\begin{equation}}
\newcommand{\ee}{\end{equation}}
\newcommand{\ben}{\begin{eqnarray*}}
\newcommand{\een}{\end{eqnarray*}}

\definecolor{darkgreen}{rgb}{0.09, 0.45, 0.27}
\newcommand{\cb}{\color{blue}}

\newtheorem{theorem}{Theorem}[section]
\newtheorem{lemma}[theorem]{Lemma}

\newtheorem{corollary}[theorem]{Corollary}
\newtheorem{proposition}[theorem]{Proposition}

\allowdisplaybreaks

\definecolor{debianred}{rgb}{0.84, 0.04, 0.33}

\begin{document}

\title{\vskip-0.3in Higher order evolution  inequalities with nonlinear convolution terms}

\author[a,*]{Roberta Filippucci}
\author[b,c,**]{Marius Ghergu}

\affil[a]{Dipartimento di Matematica e Informatica, Universit\'a degli Studi di Perugia, Via Vanvitelli 1, 06123 Perugia, Italy}
\affil[*]{Email: {\tt roberta.filippucci@unipg.it}}
\vspace{0.8in}
%\affil[]{}

\affil[b]{School of Mathematics and Statistics, University College Dublin, Belfield, Dublin 4, Ireland}
\affil[c]{Institute of Mathematics Simion Stoilow of the Romanian Academy, 21 Calea Grivitei St., 010702 Bucharest, Romania}
\affil[**]{Corresponding author, Email: {\tt marius.ghergu@ucd.ie} }

\maketitle

\begin{abstract}
We are concerned with the study of existence and nonexistence of weak solutions to 
$$
\begin{cases}
&\displaystyle \frac{\partial^k u}{\partial t^k}+(-\Delta)^m u\geq  (K\ast |u|^p)|u|^q \quad\mbox{ in } \mathbb R^N \times \mathbb R_+,\\[0.1in]
&\displaystyle \frac{\partial^i u}{\partial t^i}(x,0) = u_i(x) \,\, \text{ in } \mathbb R^N,\, 0\leq i\leq k-1,\\
\end{cases}
$$
where  $N,k,m\geq 1$ are positive integers,   $p,q>0$ and $u_i\in L^1_{\rm loc}(\mathbb{R}^N)$ for $0\leq i\leq k-1$. We  assume  that $K$ is a radial positive and continuous function which decreases in a neighbourhood of infinity. 
In the above problem, $K\ast |u|^p$ denotes  the standard convolution operation between $K(|x|)$ and $|u|^p$. 
We obtain necessary conditions on $N,m,k,p$ and $q$ such that the above problem has  solutions. Our analysis emphasizes  the role played by the sign of $\displaystyle \frac{\partial^{k-1} u}{\partial t^{k-1}}$.

\end{abstract}
\medskip


\medskip


\newpage

\section{Introduction and the main results}
Let $N,k,m\geq 1$ be positive integers. In this work we are concerned with the problem
\begin{equation}\label{main}
\begin{cases}
&\displaystyle \frac{\partial^k u}{\partial t^k}+(-\Delta)^m u\geq  (K\ast |u|^p)|u|^q \quad\mbox{ in } \mathbb R^{N+1}_+:=\mathbb R^N \times \mathbb R_+,\\[0.1in]
&\displaystyle \frac{\partial^i u}{\partial t^i}(x,0) = u_i(x) \,\, \text{ in } \mathbb R^N,\, 0\leq i\leq k-1,\\
\end{cases}
\end{equation}
where $\mathbb R_+=(0,\infty)$, $N\geq 1$,   $p,q>0$ and $u_i\in L^1_{\rm loc}(\mathbb{R}^N)$ for $0\leq i\leq k-1$. We  assume  that $K\in C(\mathbb R_+;\mathbb R_+)$ satisfies:
$$
K(|x|)\in L^1_{\rm loc}(\mathbb{R}^N)
$$
and
\medskip

\noindent $(A)\;\;$ there exists $R_0>1$ such that $\displaystyle \inf_{r\in (0,R)} K(r)=K(R)$ for all $R>R_0$. 

\medskip

In particular, condition $(A)$ above implies that $K$ is decreasing on the interval $(R_0, \infty)$.  
Typical examples of potentials $K$ satisfying the above conditions are the constant functions as well as
$$
K(r)=r^{-\alpha},\; \alpha\in (0,N)\quad\mbox{ or }\quad K(r)=r^{-\alpha}\log^{\beta}(1+r),\; \alpha\in (0,N), \beta\in \mathbb{R}, \beta>\alpha-N.
$$
The picture below illustrates some other functions $K(r)$ satisfying the hypothesis $(A)$.
\begin{figure}[h]
\begin{center}
\includegraphics*[width=.42\linewidth]{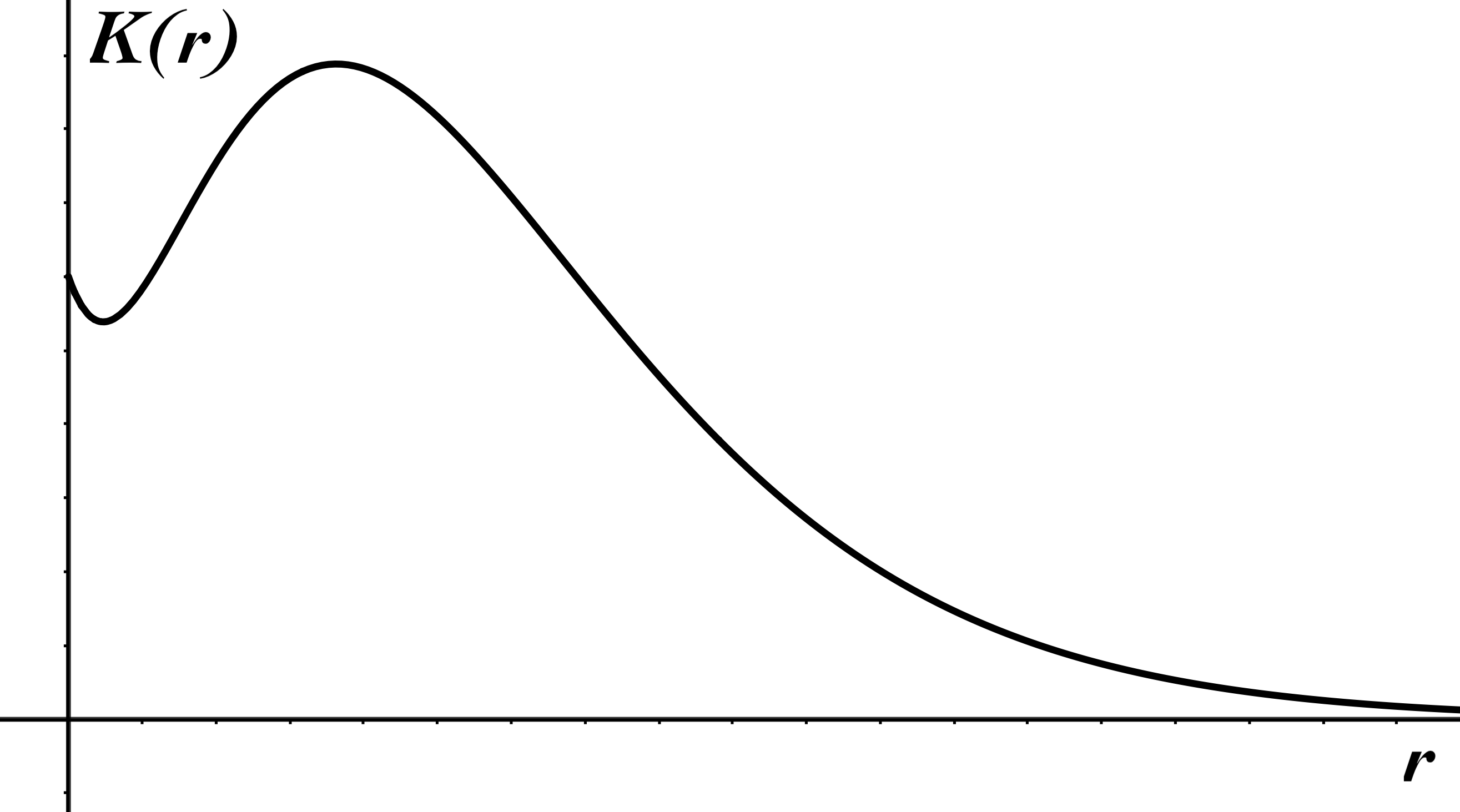} $\quad$
\includegraphics*[width=.47\linewidth]{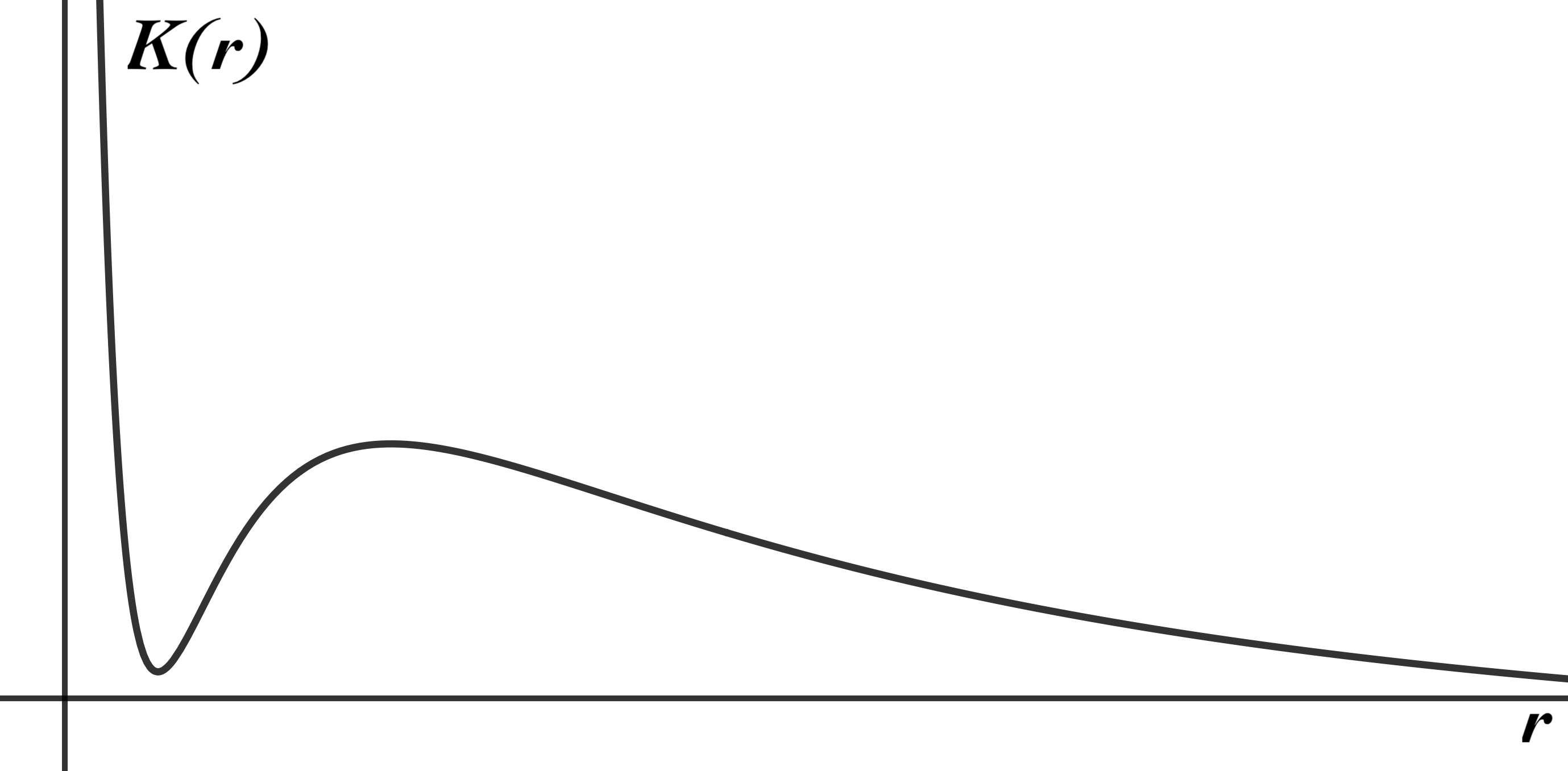}
%\caption{Some functions $K(r)$ satisfying the hypothesis $(A)$.}
\end{center}
\end{figure}

\noindent By $K\ast |u|^p$ we denote the standard convolution operator defined by
\begin{equation}\label{convl}
(K\ast |u|^p)(x,t)=\int\limits_{\mathbb R^N}K(|x-y|)|u(y,t)|^p dy\quad\mbox{ for all }(x,t)\in\mathbb R^{N+1}_+
\end{equation}
and we shall require the above integral to be finite for almost all $(x,t)\in\mathbb R^{N+1}_+$.

We are interested in  {\it weak solutions} of \eqref{main},  that is,  functions  $u\in  {L}^{p}_{\rm loc}(\mathbb R^{N+1}_+)$ such that
\begin{enumerate}
\item[(i)] $(K\ast |u|^p)|u|^q  \in  {L}^1_{\rm loc}(\mathbb R^{N+1}_+)$;

\item[(ii)] for any nonnegative test function $\varphi \in C^{\infty}_{c}(\mathbb R^{N+1}_+)$, we have
\end{enumerate}
\begin{multline} \label{weak_form}
\sum_{i=1}^{k}\! (-1)^i \!\! \int\limits_{\mathbb R^N}u_{k-i}(x)\frac{\partial^{i-1} \varphi}{\partial t^{i-1}}(x,0)dx 
+\int\limits_{{\mathbb R^{N+1}_+}} u \Big[ (-1)^k\frac{\partial^k \varphi}{\partial t^k}+(-\Delta)^m \varphi \Big]\, dxdt \\ \ge \int\limits_{{\mathbb R^{N+1}_+}} (K\ast |u|^p)|u|^q  \varphi \, dxdt.
\end{multline}

\medskip

Using a standard integration by parts it is easily seen that any classical solution of \eqref{main}
is also a weak solution. Let us point out that  condition (i)  above implies that the convolution integral in \eqref{convl} is finite for a.a. $(x,t)\in\mathbb R^{N+1}_+$. This further yields $u\in L^p_{\rm loc}(\mathbb R^{N+1}_+)$. Indeed for $R>R_0$ and 
 $x,y\in B_R$ we have  $|x-y|\le 2R$, so that by the definition of $K$ and its monotonicity we deduce
\begin{equation}\label{uLp}
\infty>(K *|u|^p)(x,t)\geq K(2R)\int\limits_{B_R}|u(y,t)|^{p}dy.
\end{equation}
Since early 1980s many research works have been devoted to the study of the prototype evolution  inequalities 
$$
\frac{\partial u}{\partial t}-\Delta u\geq  |u|^q \quad\mbox{ and } \quad 
\frac{\partial^2 u}{\partial t^2}-\Delta u\geq  |u|^q \quad\mbox{ in } \mathbb R^{N+1}_+.
$$
To the best of our knowledge, the first study of nonexistence of solutions to higher order hyperbolic inequalities is due to L. V\'eron and S.I. Pohozaev \cite{PV} related to
\begin{equation}\label{pveron}
\frac{\partial^2 u}{\partial t^2}-\mathcal{L}_m \big(\phi(u)\big)\geq  |u|^q \quad\mbox{ in } \mathbb R^{N+1}_+,
\end{equation}
where 
$$
\mathcal{L}_m v=\sum_{|\alpha|=m} D^\alpha (a_\alpha (x,t) v)\quad\mbox{ for some integer }m\geq 1
$$
and $\phi$ is a locally bounded real-valued function such that $|\phi(u)|\leq c|u|^p$, for some $c,p>0$. It is obtained in \cite{PV} that if $q>\max\{p,1\}$ and one of the following conditions hold
\begin{equation}\label{pv1}
\mbox{ either }\quad m\geq 2N\quad\mbox{ or }\quad m<2N\leq \frac{m(q+1)}{q-p},
\end{equation}
then \eqref{pveron} has no solutions in $\mathbb R^{N+1}_+$  satisfying 
\begin{equation}\label{cint}
\int_{\mathbb R^N} \frac{\partial u}{\partial t} (x,0)dx\geq 0.
\end{equation}
Condition \eqref{cint} is essential in the study of nonexistence of solutions to \eqref{pveron} and will also play an important role in the study of \eqref{main} (see Theorem \ref{thmain1} below). 
Further, if $m=2$, it is obtained in \cite{PV} that if \eqref{pv1} fails to hold then \eqref{pveron} has a positive solution. Among the results for $m=2$,  we quote also \cite{Gu}, where a weighted nonlinearity is considered. 

Another set of results that motivate our present work are due to G.G. Laptev \cite{L2002, L2003} where the following problem is studied 
\begin{equation}\label{lap}
\begin{cases}
&\displaystyle \frac{\partial^k u}{\partial t^k}-\Delta  u\geq  |x|^{-\sigma} |u|^q, u\geq 0 \quad\mbox{ in } \Omega \times \mathbb R_+,\\[0.1in]
&\displaystyle \frac{\partial^{k-1} u}{\partial t^{k-1}}(x,0) \geq 0 \,\, \text{ in } \Omega.
\end{cases}
\end{equation}
In the above, $\Omega\subset \mathbb R^N$ is either the exterior of a ball or an unbounded cone-like domain; for other results on hyperbolic inequalities in exterior domains see \cite{JS21,JSY19,MP}. We observe that solutions of \eqref{lap} are also required to satisfy \eqref{cint}. It is obtained in \cite{L2003} that if $\sigma>-2$ then the above problem has no solutions provided that
$$
1<q<q_k^*:=1+\frac{2+\sigma}{N-2+2/k}.
$$
The above exponent $q_k^*$ coincides with the Fujita-Hayakawa critical exponent if $k=1$, that is, $q_1^*=1+\frac{2+\sigma}{N}$ and with the Kato  critical exponent if $k=2$, that is, $q_2^*=1+\frac{2+\sigma}{N-1}$. 

The study of convolution terms in time-dependent partial differential equations goes back to near a century ago. Indeed, the equation
\begin{equation}\label{hartree}
i\psi_t-\Delta \psi=(|x|^{\alpha-N}\ast \psi^2)\psi\quad\mbox{in }{\mathbb R}^N, \alpha\in (0, N), N\geq 1,  
\end{equation}
was introduced  by D.H. Hartree \cite{H1,H2,H3} in 1928 for $N=3$ and $\alpha=2$ in relation to the Schr\"odinger equation in quantum physics. Nowadays, \eqref{hartree} bears the name of {\it Choquard equation}. Stationary solutions to \eqref{hartree} and its related equations have been extensively investigated from various perspectives: ground states, isolated singularities, symmetry of solutions; see e.g., \cite{FG1, GMM2021, GT2015, GT2016, GT2016book, MV2013, MV2017}.
In a different direction, G. Whithman \cite[Section 6]{W67} considered the nonlocal equation
$$
\frac{\partial u}{\partial t}+u \frac{\partial u}{\partial x}-\int_{-\infty}^{\infty}K(x-y) \frac{\partial u}{\partial x}(y,t)dy=0
$$
to study general dispersion in water waves; for recent results on nonlocal evolution equations  we refer the reader to \cite{DE22, DF21}. 
In \cite{FG2} it is studied a class of quasilinear parabolic inequalities featuring nonlocal convolution terms in the form
$$
\displaystyle \frac{\partial u}{\partial t}-\mathcal{L} u\geq  (K\ast u^p)u^q \quad\mbox{ in } \mathbb R^{N+1}_+,
$$
where $\mathcal{L}u$ is a quasilinear operators that contains as prototype model the $m$-Laplacian and the generalized mean curvature operator.

In the present work we investigate the existence and nonexistence of weak solutions to \eqref{main}. The stationary solutions to \eqref{main} were discussed recently in \cite{GMM2021}.  
Our main result concerning the inequality \eqref{main} reads as follows. 

\begin{theorem}\label{thmain1}
Assume $N,m,k\geq 1$ and $p,q>0$.

\begin{enumerate}
\item[{\rm (i)}] If $k\geq 1$ is a even integer and $q\geq 1$, then \eqref{main}  admits some positive solutions $u\in C^\infty(\mathbb{R}^{N+1}_+)$ which verify
$$
u_{k-1}=\frac{\partial^{k-1} u}{\partial t^{k-1}} (\cdot, 0)<0\quad\mbox{ in }\mathbb{R}^N.
$$
\item[{\rm (ii)}] If $p+q>2$ and
\begin{equation} \label{K_limsup}\limsup_{R\to\infty}K(R)R^{\frac{2N+2m/k}{p+q}-N+2m(1-1/k)}>0,\end{equation}
then \eqref{main} has no nontrivial solutions such that

\begin{equation} \label{int_u1}
u_{k-1}\geq 0\quad\mbox{ or }\quad u_{k-1}\in L^1(\mathbb R^N)\quad \mbox{ and } \; \int\limits_{\mathbb R^N} u_{k-1}(x)dx>0. 
\end{equation}
\end{enumerate}

\end{theorem}

Let us note that condition  \eqref{K_limsup} and the fact that $K$ is decreasing in a neighbourhood of infinity imply that 
$$
\frac{2N+2m/k}{p+q}\geq N-2m\Big(1-\frac{1}{k}\Big).
$$  
Also, under extra assumptions on $u_{k-1}$ and strengthening \eqref{K_limsup},  we can handle the case $\int_{\mathbb R^N}u_{k-1}(x)=0$ in \eqref{int_u1} for which the same conclusion as in Theorem \ref{thmain1}(ii) holds (see Proposition \ref{lo}).

The proof of Theorem \ref{thmain1} relies on nonlinear capacity estimates specifically adapted to the nonlocal setting of our problem.  More precisely, we derive integral estimates in time for the new quantity
$$
J(t)=\int_{\mathbb R^N}u^{\frac{p+q}{2}} (x,t)\varphi (x, t)dx\, , \,\,\, t\geq 0,
$$
where  $\varphi$ is a specially constructed test function  with compact support (see \eqref{test_function}).

Theorem \ref{thmain1} shows a sharp contrast in the existence/nonexistence diagram according to whether $\frac{\partial^{k-1}u}{\partial t^{k-1}}(x,0)$ has constant sign (positive or negative) on $\mathbb{R}^N$. 
To better illustrate this fact, let us discuss the case of pure powers in the potential $K(r)=r^{-\alpha}$, $\alpha\in (0,N)$ and $k=1,2$. 

Let us first consider the parabolic problem
\begin{equation}\label{mainp}
\begin{cases}
&\frac{\partial u}{\partial t} +(-\Delta)^m u\geq  (|x|^{-\alpha}\ast |u|^p)|u|^q \quad\mbox{ in } \mathbb R^{N+1}_+:=\mathbb R^N \times \mathbb R_+,\\
&u(x,0) = u_0(x) \,\, \text{ in } \mathbb R^N.
\end{cases}
\end{equation}
Note that since $k=1$ in \eqref{main}, condition \eqref{int_u1} is satisfied for all nonnegative solutions of \eqref{mainp}. 
\begin{corollary}\label{cor1} Let $N,m\geq 1$, $p,q>0$ and $\alpha\in (0, N)$.
\begin{enumerate}
\item[{\rm (i)}] If $0<\alpha<m$ and
$$
2<p+q\leq \frac{2N+2m}{N+\alpha},
$$ 
then \eqref{main} has no nontrivial nonnegative solutions;
\item[{\rm (ii)}] If $N>2m$ and
$$
\min\{p,q\}>\frac{N-\alpha}{N-2m}\quad\mbox{ and } p+q>\frac{2N-\alpha}{N-2m},
$$ 
then \eqref{mainp}  has positive solutions.
\end{enumerate}
\end{corollary}
Part (i) in the above result follows from Theorem \ref{thmain1} while part (ii) follows from \cite[Theorem 1.4]{GMM2021} where the stationary case of \eqref{mainp} was discussed.
Let us note that for $m=1$, the nonexistence of a nonnegative solution in Corollary \ref{cor1}(i) was already observed in \cite{FG2}. 

We next take $K(r)=r^{-\alpha}$, $k=2$ and $m=1$ in Theorem \ref{thmain1}. We thus consider
\begin{equation}\label{mainh}
\begin{cases}
&\displaystyle \frac{\partial^2 u}{\partial t^2} -\Delta  u\geq  (|x|^{\alpha}\ast |u|^p)|u|^q \quad\mbox{ in } \mathbb R^{N+1}_+:=\mathbb R^N \times \mathbb R_+,\\[0.1in]
&u(x,0) = u_0(x) \,\, \text{ in } \mathbb R^N,\\
&\displaystyle \frac{\partial u}{\partial t}(x,0) = u_{1}(x) \,\, \text{ in } \mathbb R^N.
\end{cases}
\end{equation}

Our result on problem \eqref{mainh} is stated below.
\begin{corollary}\label{cor2} Let $N,m\geq 1$, $p,q>0$  and $\alpha\in (0, N)$.
\begin{enumerate}
\item[{\rm (i)}] If $q\geq 1$, then \eqref{main} admits some positive solutions $u\in C^2(\mathbb{R}^{N+1}_+)$ which verify
$$
\frac{\partial u}{\partial t}(\cdot, 0)<0\quad\mbox{ in }\mathbb{R}^N.
$$
\item[{\rm (ii)}] If $0<\alpha<\min\{N,3/2\}$ and
$$
2<p+q\leq \frac{2N+1}{N+\alpha-1},
$$ 
then \eqref{mainh} has no nontrivial solutions satisfying 
\begin{equation}\label{cintt}
\frac{\partial u}{\partial t}(x,\cdot)\geq 0 \mbox{ in }\mathbb R^N \quad\mbox{ or }\quad \int\limits_{\mathbb R^N} \frac{\partial u}{\partial t}(x,\cdot)dx>0. 
\end{equation}
\item[{\rm (iii)}] If $N>2$ and
\begin{equation}\label{csol}
\min\{p,q\}>\frac{N-\alpha}{N-2}\quad\mbox{ and } \quad p+q>\frac{2N-\alpha}{N-2},
\end{equation}
then  \eqref{mainh}  admits some positive solutions which verify $\displaystyle \frac{\partial u}{\partial t}(x,\cdot)> 0$ in $\mathbb R^N$.
\end{enumerate}
\end{corollary}

The diagram of existence/nonexistence of a weak solution to \eqref{mainh} satisfying \eqref{cintt} in the $pq$-plane is given below. The light shaded region represents the region for which \eqref{mainh} admits solutions satisfying \eqref{cintt} while the dark shaded region describes the pairs $(p,q)$ for which no such solutions exist. Corollary \ref{cor2} leaves open the isue of existence and nonexistence in the white regions of the $(p>0, q>0)$ quadrant.

\begin{figure}[h]
\begin{center}
\includegraphics*[width=.9\linewidth]{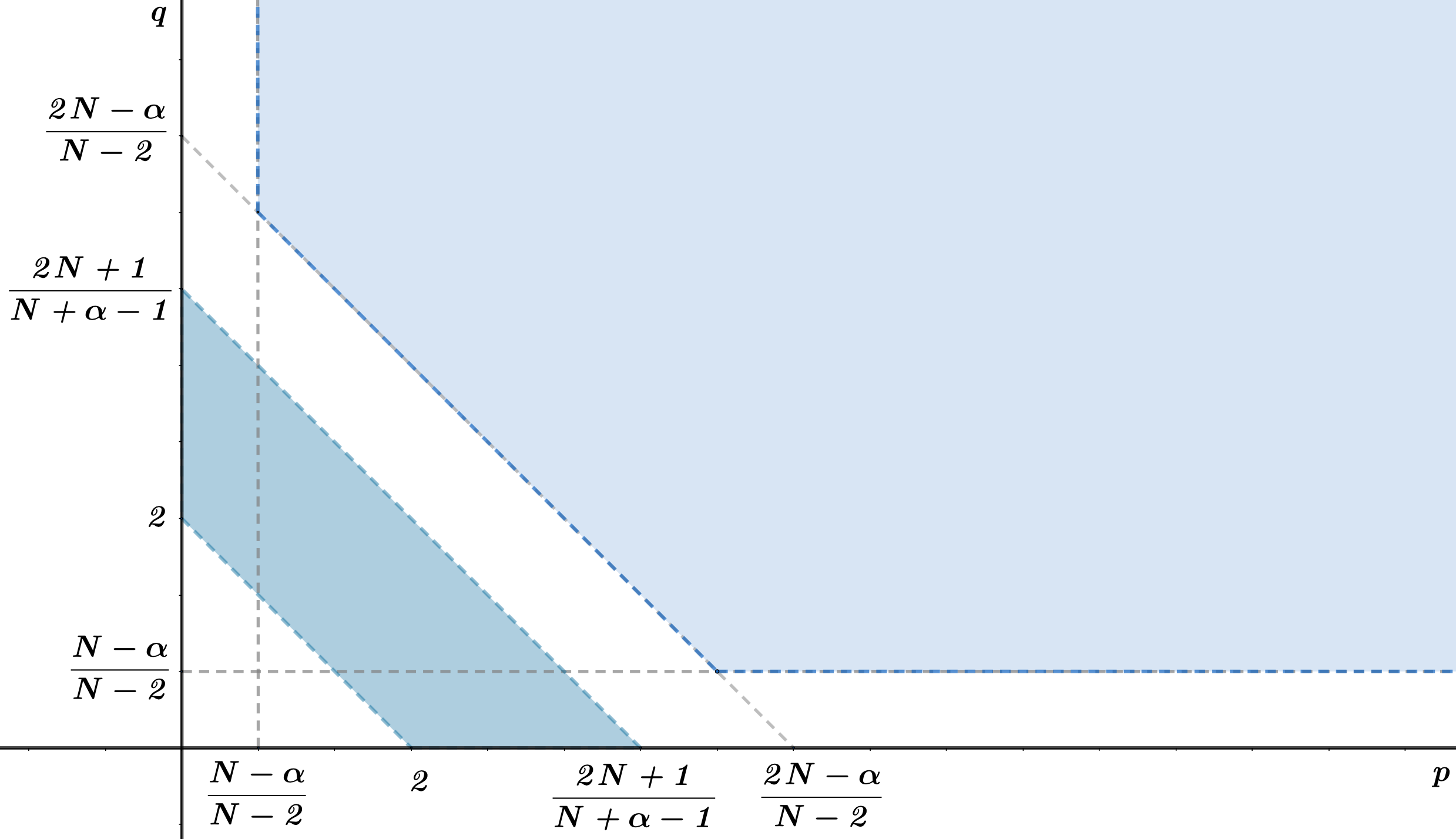} 
%\caption{Some functions $K(r)$ satisfying the hypothesis $(A)$.}
\end{center}
\end{figure}

Theorem \ref{thmain1} also applies to the case where $K\equiv 1$ for  which \eqref{main} reads
\begin{equation}\label{maink1}
\begin{cases}
&\displaystyle \frac{\partial^k u}{\partial t^k}-\Delta ^m u\geq  \Big(\int_{\mathbb R^N} |u(y)|^pdy\Big) |u|^q \quad\mbox{ in } \mathbb R^{N+1}_+,\\[0.1in]
&\displaystyle \frac{\partial^i u}{\partial t^i}(x,0) = u_i(x) \,\, \text{ in } \mathbb R^N,\, 0\leq i\leq k-1.\\

\end{cases}
\end{equation}

 A similar conclusion to Corollary \ref{cor1} and Corollary \ref{cor2} (in which we let $\alpha=0$) hold for \eqref{maink1}.

We can further employ the ideas in the study of \eqref{main} to the case of systems of type

\begin{equation}\label{mains}
\begin{cases}
&\displaystyle \frac{\partial^{k} u}{\partial t^{k}}+(-\Delta)^m u\geq  (K\ast |v|^p)|v|^q \quad\mbox{ in } \mathbb R^{N+1}_+=\mathbb R^N \times \mathbb (0,\infty),\\[0.1in]
&\displaystyle \frac{\partial^{k} v}{\partial t^{k}}+(-\Delta)^m v\geq  (L\ast |u|^n)|u|^s \quad\mbox{ in } \mathbb R^{N+1}_+,\\[0.1in]
&\displaystyle \frac{\partial^i u}{\partial t^i}(x,0) = u_i(x) \,\, \text{ in } \mathbb R^N,\, 0\leq i\leq k-1,\\[0.1in]
& \displaystyle \frac{\partial^i v}{\partial t^i}(x,0) = v_i(x) \,\, \text{ in } \mathbb R^N,\, 0\leq i\leq k-1,\\
\end{cases}
\end{equation}
where $N,m,k\geq 1$,   $p,q,n,s>0$. We assume $u_i, v_i\in L^1_{\rm loc}(\mathbb{R}^N)$, $0\leq i\leq k-1$ and that $K,L$ satisfy condition $(A)$ for some $R_0>1$.

A pair $(u, v)\in  {L}^{n}_{\rm loc}(\mathbb R^{N+1}_+)\times {L}^{p}_{\rm loc}(\mathbb R^{N+1}_+)$ is a {\it weak solution} of \eqref{mains} if:
\begin{enumerate}
\item[(i)] $(L\ast u^n)u^s, (K\ast v^p)v^q   \in  {L}^1_{\rm loc}(\mathbb R^{N+1}_+)$;

\item[(ii)]  for any nonnegative test function  $\varphi \in C^{\infty}_{c}(\mathbb R^N\times \mathbb R),$ we have
\end{enumerate}
\begin{equation} \label{weak_forms1}
\left\{
\begin{aligned}
\sum_{i=1}^{k}\! (-1)^i \!\! \int\limits_{\mathbb R^N}u_{k-i}(x)\frac{\partial^{i-1} \varphi}{\partial t^{i-1}}(x,0)dx +\int\limits_{{\mathbb R^{N+1}_+}} u \Big[ (-1)^k\frac{\partial^k \varphi}{\partial t^k}& +(-\Delta)^m \varphi \Big]\, dxdt \\& \ge \int\limits_{{\mathbb R^{N+1}_+}} (K\ast |v|^p)|v|^q  \varphi \, dxdt,\\
\sum_{i=1}^{k}\! (-1)^i \!\! \int\limits_{\mathbb R^N}v_{k-i}(x)\frac{\partial^{i-1} \varphi}{\partial t^{i-1}}(x,0)dx  
+\int\limits_{{\mathbb R^{N+1}_+}} v  \Big[ (-1)^k\frac{\partial^k \varphi}{\partial t^k} & +(-\Delta)^m \varphi \Big]\, dxdt \\& \ge \int\limits_{{\mathbb R^{N+1}_+}} (L\ast |u|^n)|u|^s  \varphi \, dxdt.
\end{aligned}
\right.
\end{equation}

Our main result concerning \eqref{weak_forms1} is stated below.
\begin{theorem}\label{thmainsystem}
Assume $\min\{p+q, n+s\}>2$. If either
\begin{equation} \label{Ksup1}
\limsup_{R\to\infty}K(R)L(R)^{ \frac{1}{n+s}} R^{ \frac{2N+2m/k}{(n+s)(p+q)}+\frac{N+2m}{n+s}-N+2m\big(1-\frac{1}{k}\big) }>0,
\end{equation}
or
\begin{equation} \label{Ksup2}
\limsup_{R\to\infty}K(R)^{\frac{1}{p+q} }L(R)  R^{ \frac{2N+2m/k}{(n+s)(p+q)}+\frac{N+2m}{p+q}-N+2m\big(1-\frac{1}{k}\big) }>0,
\end{equation}
then \eqref{weak_forms1} has no nontrivial  solutions such that
\begin{equation} \label{int_uv}
u_{k-1},v_{k-1}\geq 0\quad\mbox{ or }\quad u_{k-1}, v_{k-1}\in L^1(\mathbb R^N)\quad \mbox{ and } \; 
\int_{\mathbb R^N} u_{k-1}(x)dx, \, \int_{\mathbb R^N} v_{k-1}(x)dx>0. 
\end{equation}

\end{theorem}
The next sections contain the proofs of the above results.

\section{Proof of Theorem \ref{thmain1}}

{\bf Proof of Theorem 1} (i)  Let $\gamma>0$ be such that $p\gamma>N$ and define
\begin{equation}\label{equu}
u(x,t)=e^{-Mt}v(x)^{-\gamma/2}\quad\mbox{ with }\quad v(x)=1+|x|^2,
\end{equation}
where $M>1$ will be specified later. Since $K>0$ is continuous in $\mathbb R^+$ and  decreasing in a neighbourhood of infinity (by condition $(A)$) it follows that 
$$
\sup_{[1,\infty)}K=\max_{[1,\infty)}K<\infty.
$$
Furthermore, we have 
\begin{equation}\label{kest1}
\begin{aligned}
\big(K\ast v^{-p\gamma/2}\big)(x)&=\int_{\mathbb R^N}K(|z|)(1+|z-x|^2)^{-p\gamma/2}dz\\
&\leq  \int_{B_1}K(|z|)dz+\big(\max_{[1,\infty)}K\big) \int_{\mathbb R^N\setminus B_1 }(1+|z-x|^2)^{-p\gamma/2}dz\\
&\leq  \int_{B_{1}}K(|z|)dz+\big(\max_{[1,\infty)}K\big) \int_{\mathbb R^N} (1+|z-x|^2)^{-p\gamma/2}dz\\
&=  \int_{B_{1}}K(|z|)dz+\big(\max_{[1,\infty)}K\big) \int_{\mathbb R^N} (1+|y|^2)^{-p\gamma/2}dy\\
&\leq C(K,p,\gamma), 
\end{aligned}
\end{equation}
since  $p\gamma>N$.

Further, a direct calculation shows that $-\Delta \big(v^{-\gamma/2}\big)=c_1 v^{-\gamma/2-1}+c_2 v^{-\gamma/2-2}$ in $\mathbb R^N$, where $c_1, c_2$ are real coefficients depending on $N$ and $\gamma$. Hence, an induction argument yields
$$ (-\Delta)^m \bigl(v^{-\gamma/2}\bigr)=v^{-2m-\gamma/2}\sum_{j=0}^{m}c_j v^{j}\quad \mbox{ in }\mathbb R^N,$$
where $c_j=c_j(\gamma,N,m)\in\mathbb R$. Thus, the function $u$ given by \eqref{equu} satisfies
\begin{equation}\label{kest2}
\frac{\partial^k u}{\partial t^k} +(-\Delta)^m u=e^{-Mt}\Big((-1)^kM^k v^{2m}+\sum_{j=0}^{m}c_jv^j\Big) v^{-2m-\gamma/2}\quad\mbox{ in }\mathbb R^{N+1}_+,
\end{equation}
where $c_j\in \mathbb{R}$ are independent of $M$.  Using the fact that $k$ is an even integer, by taking $M>1$ large, we may ensure that 
\begin{equation}\label{kest3}
M^k v^{2m}+\sum_{j=0}^{m}c_jv^j\geq C(K,p,\gamma)v^{2m} \quad\mbox{ in }\mathbb R^N,
\end{equation}
where $C(K,p,\gamma)>0$ is the constant from \eqref{kest1}.
Now, combining \eqref{kest1}-\eqref{kest3} and using that $u^p e^{Mpt}=v^{-p\gamma/2}$, we deduce
$$
\begin{aligned}
\frac{\partial^k u}{\partial t^k}-\Delta^m u &\geq C(K,p,\gamma) e^{-Mt}v^{-\gamma/2}\\
&\geq C(K,p,\gamma) e^{-(p+q)Mt}v^{-\gamma/2}\\
&\geq \big(K\ast u^{p}\big) e^{-qMt}v^{-\gamma/2}\\
&\geq \big(K\ast u^{p}\big) \Big(e^{-Mt}v^{-\gamma/2})^q\\
&\geq \big(K\ast u^{p}\big) u^q\quad\mbox{ in }\mathbb R^{N+1}_+,
\end{aligned}
$$
being  $q\ge1$, which concludes the proof of part (i). Since $k$ is even, one has
 $$u_{k-1}(x)=\frac{\partial^{k-1} u}{\partial t^{k-1}}(x,0) =(-1)^{k-1}M^kv(x)<0.$$

(ii) Assume that $p+q>2$, \eqref{K_limsup} and \eqref{int_u1}  hold and that  \eqref{main} admits a weak solution $u$. We shall first construct a suitable test function $\varphi$ for \eqref{weak_form}. To do so,  take a standard cut off function $\varrho \in C^\infty_c(\mathbb R)$ 
such that:
\begin{itemize}
\item $\varrho= 1$ in $(0,1)$, $\varrho= 0$  in $(2,\infty)$;
\item  $0\le \varrho\le 1$, $\text{supp}\varrho \subseteq[0,2]$.
\end{itemize}

Now take $R > 0$,  $\gamma>0$ to be precised later, $\kappa\ge1$ sufficiently large and consider the function
\begin{equation} \label{test_function}
\varphi(x) =\varrho^\kappa \biggl(\frac{|x|}{R}\biggr)\varrho^\kappa\biggl(\frac{t}{R^{\gamma}} \biggr) 
\quad\text{in}\quad \mathbb R^{N+1}_+ .
\end{equation}
Clearly 
\begin{equation} \label{supp_rho}\text{supp\,}\varphi\subset B_{2R}\times [0,2R^\gamma)\subset \mathbb R^{N+1}_+,
\end{equation}
and
\begin{equation}\label{supp_der_varphi}\text{supp\,}\frac{\partial^k\varphi}{\partial t^k}\subset  B_{2R}\times [R^\gamma,2R^\gamma)\subset \text{supp\,}\varphi, 
\qquad \text{supp\,}\Delta\varphi\subset (B_{2R}\setminus B_R)\times [0,2R^\gamma)\subset \text{supp\,}\varphi.
\end{equation}
As in  \cite[Lemma 3.1]{FL},  we have the following estimates:
\begin{proposition}
Let $\varphi$ be defined in \eqref{test_function}. Then for $\varsigma>1$ and $\kappa\ge 2m\varsigma$ one has
\begin{equation} \label{prop1_test}
\int_{\text{\rm supp}\, \varphi} \frac{1}{\varphi^{\varsigma-1}}    \Big|\frac{\partial^i \varphi}{\partial t^i}\Big|^\varsigma dxdt\le C R^{N+\gamma-i\gamma\varsigma },\quad i\geq 1,
\end{equation}
\begin{equation} \label{prop2_test}
\int_{ \text{\rm supp}\,\varphi}\frac{|\Delta^m\varphi|^\varsigma}{\varphi^{\varsigma-1}}dxdt \le C R^{N+\gamma-2m\varsigma },
\end{equation}
where $C$ is a positive constant changing from line to line.
\end{proposition}

Since  $\frac{\partial^i \varphi}{\partial t^i}(x,0)=0$ in $\mathbb R^N$ for all $i\geq 1$,  from \eqref{weak_form} we deduce 
\begin{equation} \label{weak_0}\begin{aligned}
\int_{\mathbb R^N}u_{k-1}(x)\varphi(x,0)\, &dx+\int_{\mathbb R^{N+1}_+} (K\ast |u|^p)|u|^q  \varphi \, dxdt \\&\le 
\int_{\mathbb R^{N+1}_+}|u| \Big|\frac{\partial^k \varphi}{\partial t^k} \Big|\, dxdt+
\int_{\mathbb R^{N+1}_+} |u| |\Delta^m\varphi|\, dx dt.
\end{aligned}\end{equation}
Observe that by \eqref{int_u1} we have $u_{k-1}\geq 0$ or $u_{k-1}\in L^1(\mathbb R^N)$ and $\int_{\mathbb R^N} u_{k-1}(x)dx>0$. In the latter case, from $u_{k-1}\in L^1_{loc}(\mathbb R^N)$, we deduce, using \eqref{test_function},

\begin{equation} \label{sign_u1_rho}
\begin{aligned}
\lim_{R\to\infty} \int_{\mathbb R^N}u_{k-1}(x)\varphi(x,0)dx&=\lim_{R\to\infty} \int_{\mathbb R^N}u_{k-1}(x)\varrho^{\kappa}\biggl(\frac{|x|}{R}\biggr)dx=\int_{\mathbb R^N}u_{k-1}(x)dx.
\end{aligned}\end{equation}

Thus, from \eqref{int_u1} we deduce  that for large $R>0$ we have
\begin{equation}\label{hyp_ge0}
\int_{\mathbb R^N}u_{k-1}(x)\varphi(x,0)dx\geq 0, 
\end{equation}
case in which \eqref{weak_0} yields
\begin{equation} \label{weak}
\int_{\mathbb R^{N+1}_+} (K\ast |u|^p)|u|^q  \varphi \, dxdt \le 
\int_{\mathbb R^{N+1}_+}|u| \Big|\frac{\partial^k \varphi}{\partial t^k} \Big|\, dxdt+
\int_{\mathbb R^{N+1}_+} |u| |\Delta^m\varphi|\, dx dt,
\end{equation}
provided $R>0$ in the definition of $\varphi$ (see \eqref{test_function}) is large enough.

An important tool  of our approach is the following result.

\begin{lemma}\label{lk} 
For almost all $t\geq 0$ we have $u(\cdot, t)\in L^{\frac{p+q}{2}}_{loc}(\mathbb R^N)$ and 
\begin{equation}\label{J^2_1}
\int_{\mathbb R^N}(K\ast |u|^p)|u(x,t)|^q  \varphi(x,t) dx\ge K(4R) J^2(t) \quad\mbox{ for all }\;R\geq R_0,
\end{equation}
where
\begin{equation}\label{JJ}
J(t)=\int_{\mathbb R^N}|u(x,t)|^{\frac{p+q}2} \varphi(x,t) dx.
\end{equation}
\end{lemma}

\begin{proof} First  note that 
for $x,y\in B_{2R}$ one has $|x-y|\le 4R$, so that, thanks to the monotonicity of $K$, 
$$\int_{\mathbb R^N}K(|x-y|)|u(y)|^p dy\ge \int_{B_{2R}}K(|x-y|)|u(y)|^p dy\ge 
K(4R)\int_{B_{2R}}|u(y)|^p  dy.$$
Hence
\begin{equation}\label{summ}\int_{\mathbb R^N}(K\ast |u|^p)|u(x,t)|^q \varphi(x,t) dx\ge K(4R) \int_{\mathbb R^N}\int_{\mathbb R^N}
|u(y,t)|^p\varphi(y,t) |u(x,t)|^q \varphi(x,t) dxdy,
\end{equation}
where we have used that $\varphi\le1$ and that $\varphi(\cdot, t)\equiv0$ outside of $B_{2R}$ for all $t$.

Furthermore,  by H\"older's inequality we have
$$
\begin{aligned}
\biggl(\iint_{\mathbb R^N\times \mathbb R^N}|u(y,t)|^p\varphi(y,t) & |u(x,t)|^{q} \varphi(x,t)  \, dx \, dy\biggr) ^2\\
=\; &\biggl(\iint_{\mathbb R^N\times \mathbb R^N}|u(y,t)|^p\varphi(y,t)  |u(x,t)|^{q} \varphi(x,t)  \, dx \, dy\biggr)\\
\; &\cdot \biggl(\iint_{\mathbb R^N\times \mathbb R^N}|u(x,t)|^p\varphi(x,t)  |u(y,t)|^{q} \varphi(y,t) \, dx \, dy\biggr) \\
\geq \; &
\biggl(\iint_{\mathbb R^N\times \mathbb R^N} |u(x,t)|^{\frac{p+q}{2}} |u(y,t)|^{\frac{p+q}{2}}\varphi(x,t)  \varphi(y,t) \, dx \, dy\biggr)^2\\
=\; & \biggl(\int_{\mathbb R^N} |u(x,t)|^{\frac{p+q}{2}} \varphi(x,t)  \, dx \biggr)^4=J(t)^4,
\end{aligned}
$$
which, by \eqref{summ} and part (i) in the definition of a solution,  yields $u(\cdot, t)\in L^{\frac{p+q}{2}}_{loc}(\mathbb R^N)$ and \eqref{J^2_1}.
\end{proof}

Inserting \eqref{J^2_1} into \eqref{weak} we find
\begin{equation} \label{weak_J^2}
K(4R)\int_0^{2R^\gamma}J^2(t)dt \le \int_{\mathbb R^{N+1}_+}|u|\left( \Big|\frac{\partial^k \varphi}{\partial t^k}\Big|+|\Delta^m \varphi|\right)\, dx dt.  
\end{equation}
We next estimate the integral term on the right hand side of \eqref{weak_J^2}. Using H\"older's inequality, we have
\begin{equation}\label{rho_tt<}\begin{aligned}\int_{\mathbb R^{N+1}_+} |u|\Big|\frac{\partial^k \varphi}{\partial t^k}\Big|dxdt&\le  \biggl(\int_{\text{supp} (\frac{\partial^k \varphi}{\partial t^k})} |u|^{\frac{p+q}{2}} \varphi \, dxdt\biggr)^{\frac{2}{p+q}}
\cdot \biggl(\int_{ \text{supp} (\frac{\partial^k \varphi}{\partial t^k})}\varphi^{-\frac{2}{p+q-2}} \Big|\frac{\partial^k \varphi}{\partial t^k}\Big|^{\frac{p+q}{p+q-2}}\biggr)^{\frac{p+q-2}{p+q}}
\\&\le
C \biggl(\int_{\text{supp} (\frac{\partial^k \varphi}{\partial t^k})} |u|^{\frac{p+q}{2}} \varphi \, dxdt\biggr)^{\frac{2}{p+q}} R^{-k\gamma+(N+\gamma)\frac{p+q-2}{p+q}},
\end{aligned}\end{equation}
where in the last inequality we have used \eqref{prop1_test} with $\varsigma=\frac{p+q}{p+q-2}$,  $i=k$ and thanks to \eqref{supp_der_varphi}. Note that 
\begin{equation}\label{sup1}
\text{supp} (\frac{\partial^k \varphi}{\partial t^k})\subset B_{2R}\times [R^\gamma, 2R^\gamma].
\end{equation} 
By H\"older's inequality we find
\begin{equation}\label{estimat1}
\begin{aligned}\int_{\mathbb R^{N+1}_+}|u|\Big|\frac{\partial^k \varphi}{\partial t^k}\Big|dxdt&\le  
C R^{-k\gamma+(N+\gamma)\frac{p+q-2}{p+q}} \biggl(\int_{R^\gamma}^{2R^\gamma}\Big(\int_{\mathbb R^N} |u|^{\frac{p+q}{2}} \varphi \, dx\Big)\,dt\biggr)^{\frac{2}{p+q}} \\
&\le C  R^{-k\gamma+(N+\gamma)\frac{p+q-2}{p+q}+\frac{\gamma}{p+q}} \biggl(\int_{R^\gamma}^{2R^\gamma}\Big(\int_{\mathbb R^N} |u|^{\frac{p+q}{2}} \varphi \, dx\Big)^2\,dt\biggr)^{  \frac{1}{p+q} }.
\end{aligned}\end{equation}

Similarly, by H\"older's inequality, \eqref{prop2_test} with $\varsigma=\frac{p+q}{p+q-2}$ and  \eqref{supp_der_varphi}, we find
\begin{equation}\label{Delta_m_rho<}\begin{aligned}
\int_{\mathbb R^{N+1}_+}|u||\Delta^m\varphi|dxdt&\le  \biggl(\int_{\text{supp} (\Delta^m\varphi)} |u|^{\frac{p+q}{2}} \varphi \, dxdt\biggr)^{\frac{2}{p+q}}
\cdot \biggl(\int_{\text{supp} (\Delta^m\varphi)}\varphi^{-\frac{2}{p+q-2}}  |\Delta^m\varphi|^{\frac{p+q}{p+q-2}}\biggr)^{\frac{p+q-2}{p+q}}
\\&\le
C \biggl(\int_{\text{supp} (\Delta^m\varphi)} |u|^{\frac{p+q}{2}} \varphi \, dxdt\biggr)^{\frac{2}{p+q}} R^{-2m+(N+\gamma)\frac{p+q-2}{p+q}}.
\end{aligned}\end{equation}
Since
\begin{equation}\label{sup2}
\text{supp} (\Delta^m \varphi )\subset (B_{2R}\setminus B_R)  \times [0, 2R^\gamma],
\end{equation}
a new application of H\"older's inequality in the above estimate yields
\begin{equation}\label{estimat2}
\begin{aligned}
\int_{\mathbb R^{N+1}_+} |u| |\Delta^m\varphi|dxdt&\le  
C R^{-2m+(N+\gamma)\frac{p+q-2}{p+q}} \biggl(\int_0^{2R^\gamma}\Big( \int_{B_{2R}\setminus B_R} |u|^{\frac{p+q}{2}} \varphi \, dx \Big)dt\biggr)^{\frac{2}{p+q}} \\
&\le C R^{-2m+(N+\gamma)\frac{p+q-2}{p+q}+\frac{\gamma}{p+q}} \biggl(\int_0^{2R^\gamma}\Big( \int_{B_{2R}\setminus B_R} |u|^{\frac{p+q}{2}} \varphi \, dx \Big)^2dt\biggr)^{ \frac{1}{p+q}}.
\end{aligned}\end{equation}
Comparing the powers of $R$ in \eqref{estimat1} and \eqref{estimat2} we are led to choose $\gamma>0$ so that $k\gamma=2m$, that is $\gamma=2m/k$. Also,
$$
-2m+(N+\gamma)\frac{p+q-2}{p+q}+\frac{\gamma}{p+q}=N-2m\Big(1-\frac{1}{k}\Big) -\frac{2N+\frac{2m}{k}}{p+q}.
$$
Thus, \eqref{weak_J^2} together with \eqref{estimat1} and \eqref{estimat2} yield
\begin{equation} \label{estimat3}
\begin{aligned}
K(4R)  & \int_0^{2R^{\gamma}} J^2(t)dt \le   CR^{N-2m\big(1-\frac{1}{k}\big) -\frac{2N+2m/k}{p+q}}\times \\
&\times  \left[  \biggl(\int_{R^\gamma}^{2R^\gamma}\Big(\int_{\mathbb R^N} |u|^{\frac{p+q}{2}} \varphi \, dx\Big)^2\,dt\biggr)^{\frac{1}{p+q} } + 
\biggl(\int_0^{2R^\gamma}\Big( \int_{B_{2R}\setminus B_R} |u|^{\frac{p+q}{2}} \varphi \, dx \Big)^2dt\biggr)^{\frac{1}{p+q} }\right].
\end{aligned}
\end{equation}
Using \eqref{JJ}, the above estimate implies
\begin{equation}\label{weak_weak}
K(4R)  \int_0^{2R^{\gamma}}\! J^2(t)dt \le CR^{N-2m\big(1-\frac{1}{k}\big) -\frac{2N+2m/k}{p+q}} \Big(\int_0^{2R^\gamma} J^2 (t)dt \Big)^{\frac{1}{p+q} }
\end{equation}
which further yields
\begin{equation} \label{ineq2}
\biggl(\int_0^{2R^\gamma}J^2(t)dt \biggr)^{\frac{p+q-1}{p+q} }\le C\frac 1{K(4R)R^{\frac{2N+2m/k}{p+q}-N+2m(1-1/k)} }.
\end{equation}
Let $\{R_j\}_j$ be a divergent sequence that achieves the limsup in \eqref{K_limsup}, namely
$$K(4R_j)R_j^{\frac{2N+2m/k}{p+q}-N+2m(1-1/k)}\to\ell>0\quad\text{as}\quad j\to\infty.$$
Passing to a subsequence we may assume $R_j>2R_{j-1}$ for all $j>1$.

If $\ell=\infty$, we can pass to the limit in  \eqref{ineq2}, by replacing $R$ with $R_j$, to raise
$\int_0^\infty J^2(t)dt =0$ from where $J\equiv0$ in $\mathbb R^+$ and then, by the definition of $J$,  
$$\int_{B_R\times(0,\infty)}|u|^{(p+q)/2}dxdt=0 \quad\text{for all}\quad  R>0,$$ namely $u\equiv0$  in $\mathbb R^{N+1}_+$ as required.

If $\ell\in(0,\infty)$, then \eqref{ineq2} shows that $J\in L^2(0,\infty)$. Using this fact we infer that
\begin{equation}\label{eqeq}
\int_{R_j^m}^{2R_j^{m}}\!\! \Big(\int_{\mathbb R^N} |u(x,t)|^{\frac{p+q}{2}}\varphi(x,t) dx\Big)^2 dt\to 0\,,\;
\int_{0}^{2R_j^{m}} \!\!\Big(\int_{B_{2R_j}\setminus B_{R_j} } |u(x,t)|^{\frac{p+q}{2}}\varphi(x,t) dx\Big)^2 dt\to 0, 
\end{equation}
as $j\to \infty$. 
Indeed, the first limit in \eqref{eqeq} follows immediately by the fact that $J\in L^2(0,\infty)$. To check the second limit in \eqref{eqeq} we observe that
$$
\begin{aligned}
\infty>\int_0^\infty \Big(\int_{\mathbb R^N} |u(x,t)|^{\frac{p+q}{2}}dx\Big)^2 dt &\geq 
\int_0^\infty \Big(\sum_{j\geq 1}\int_{B_{2R_j}\setminus B_{R_j}} |u(x,t)|^{\frac{p+q}{2}}dx\Big)^2 dt\\
&\geq \int_0^\infty \sum_{j\geq 1}\Big(\int_{B_{2R_j}\setminus B_{R_j}} |u(x,t)|^{\frac{p+q}{2}}dx\Big)^2 dt\\
&= \sum_{j\geq 1}\int_0^\infty \Big(\int_{B_{2R_j}\setminus B_{R_j}} |u(x,t)|^{\frac{p+q}{2}}dx\Big)^2 dt.
\end{aligned}
$$
The convergence of the last series in the above estimate implies the second part of \eqref{eqeq}.
Using this fact in \eqref{estimat3} we find 
$$
\biggl(\int_0^{2R_j^\gamma}J^2(t)dt \biggr)^{\frac{p+q-2}{p+q}} \le C\frac 1{K(4R_j)R_j^{\frac{2N+2m/k}{p+q}-N+2m(1-1/k)} }\, o(1) \quad\mbox{ as  }j\to \infty.
$$
Again, by letting $R_j\to\infty$, we can conclude, also when $\ell\in(0,\infty)$, that $\int_0^\infty J^2(t)dt =0$, namely
$u\equiv0$  in $\mathbb R^{N+1}_+$.\qed

\medskip

A similar argument allows us to treat the case $\int_{\mathbb R^N} u_{k-1}(x)dx=0$. In this case we need to be more precise on the behavior of $u_{k-1}$. 

\begin{proposition}\label{lo}
 Let  $\varrho \in C^\infty_c(\mathbb R)$ 
satisfy $\text{\rm supp }\varrho \subseteq[0,2]$, $0\le \varrho\le 1$ and   $\varrho= 1$ in $(0,1)$.

Assume that for some $\kappa\geq 2m$ we have 
\begin{equation} \label{intuu}
\int_{\mathbb R^N}u_{k-1}(x)\varrho^\kappa \left(\frac{x}{R}\right)dx=O(R^{-\beta})\quad \mbox{ as }R\to \infty,\quad\mbox{ for some }\beta>0.
\end{equation}
If 
\begin{equation}\label{conb}
\limsup_{R\to\infty}K(R)R^{\min\{\beta, \frac{2N+2m/k}{p+q}-N+2m(1-1/k)\} }>0,
\end{equation}
then \eqref{main} has no nontrivial solutions.
\end{proposition}
\begin{proof}
 With the help of the above $\varrho$ we construct the test function  $\varphi$ as defined in \eqref{test_function}. Thus,  \eqref{hyp_ge0} is no more in force and \eqref{weak_0} changes to
$$
%\begin{aligned}
\int_{\mathbb R^{N+1}_+} (K\ast |u|^p)|u|^q  \varphi \, dxdt\le C R^{-\beta}+
\int_{\mathbb R^{N+1}_+}|u| \Big|\frac{\partial^k \varphi}{\partial t^k} \Big|\, dxdt+
\int_{\mathbb R^{N+1}_+} |u| |\Delta^m\varphi|\, dx dt .
%\end{aligned}
$$
Consequently using \eqref{estimat1} and \eqref{estimat2},  with $\gamma=2k/m$, and  \eqref{J^2_1}, then  the above inequality gives
\begin{equation}\label{wea}
K(4R)  \int_0^{2R^{\gamma}}\! J^2(t)dt \le CR^{-\beta}+CR^{N-2m\big(1-\frac{1}{k}\big) -\frac{2N+2m/k}{p+q}} \Big(\int_0^{2R^\gamma} J^2 (t)dt \Big)^{\frac{\cb 1}{p+q}}.
\end{equation}
Let $\sigma=\frac{2N+2m/k}{p+q}-N+2m(1-1/k) $. 
By Young's inequality we have
$$
\begin{aligned}
CR^{N-2m\big(1-\frac{1}{k}\big) -\frac{2N+2m/k}{p+q}} \Big(\int_0^{2R^\gamma} J^2 (t)dt \Big)^{\frac{\cb 1}{p+q}}&=CR^{-\sigma} \Big(\int_0^{2R^\gamma} J^2 (t)dt \Big)^{\frac{\cb 1}{p+q}}\\
&=CR^{-\sigma} K(4R)^{-\frac{\cb 1}{p+q}}\Big(K(4R)\int_0^{2R^\gamma} J^2 (t)dt \Big)^{\frac{\cb 1 }{p+q}}\\
&\leq \frac{K(4R)}{2}\int_0^{2R^\gamma} J^2 (t)dt +CR^{-\frac{\sigma(p+q)}{p+q-\cb 1}}K(4R)^{-\frac{\cb 1}{p+q-\cb 1}}.
\end{aligned}
$$
Using this last inequality into \eqref{wea} we find
$$
\frac{K(4R)}{2}\int_0^{2R^\gamma} J^2 (t)dt \leq CR^{-\beta}+CR^{-\frac{\sigma(p+q)}{p+q-\cb 1}}K(4R)^{-\frac{\cb 1}{p+q-\cb 1}},
$$
that is,
$$
\int_0^{2R^\gamma} J^2 (t)dt \leq \frac{C}{R^{\beta}K(4R)}+\frac{C}{\Big(R^{\sigma} K(4R)\Big)^{\frac{p+q}{p+q-\cb 1}}}.
$$
Now, in virtue of \eqref{conb} we may let $R\to \infty$ in the above estimate to deduce $J=0$ and then $u\equiv 0$. 
\end{proof}

\section{Proof of Corollary \ref{cor2}}
Part (i) and (ii) in Corollary \ref{cor2} follow directly from Theorem \ref{thmain1}.  

(iii) Let $p,q,\alpha$ satisfy  \eqref{csol} which we may write as
$$
(N-2)\min\{p,q\}>N-\alpha\quad\mbox{ and }\quad (N-2)(p+q-1)>N+2-\alpha.
$$
Thus, we may choose $\beta\in (0, N-2)$ such that
\begin{equation}\label{ga1}
\beta \min\{p,q\}>N-\alpha\,, \quad \beta p\neq N  \quad\mbox{ and }\quad \beta(p+q-1)>N+2-\alpha
\end{equation}
Set $w(x,t)=(1+t)^{-a}+1+|x|^2$ where $a>0$ will be precised later and $u(x,t)=w^{-\beta/2}(x,y)$.
Then
\begin{equation}\label{ga2}
\begin{aligned}
-\Delta u&= \beta N w^{-\beta/2-1}-\beta(\beta+2)w^{-\beta/2-2}|x|^2\\&=\beta\Big[N-(\beta+2)\frac{|x|^2}{w}\Big]w^{-\beta/2-1}\ge \beta(N-\beta-2)w^{-(\beta+2)/2} \quad\mbox{ in }\mathbb R^{N+1}_+.
\end{aligned}
\end{equation}
Also,
\begin{equation}\label{ga3}
\begin{aligned}
\frac{\partial^2 u}{\partial t^2}& =-\frac{a(a+1)\beta}{2}(1+t)^{-a-2}w^{-(\beta+2)/2}+\frac{a^2\beta(\beta+2)}{4}(1+t)^{-2a-2}w^{-(\beta+2)/2-1}\\
&=\frac{a\beta}2 w^{-(\beta+2)/2}(1+t)^{-a-2}\biggl[-a-1+\frac a2(\beta+2)\frac{(1+t)^a}{w}\biggr]
\\
&\geq  -\frac{a(a+1)\beta}{2}w^{-(\beta+2)/2} \quad\mbox{ in }\mathbb R^{N+1}_+.
\end{aligned}
\end{equation}
Combining now \eqref{ga2} and \eqref{ga3} we have
$$
\frac{\partial^2 u}{\partial t^2}-\Delta u\geq \beta \Big( N-2-\beta-\frac{a(a+1)}{2}\Big) w^{-(\beta+2)/2}
\quad\mbox{ in }\mathbb R^{N+1}_+.
$$
Thus, by letting $a>0$ small enough, the above estimate yields
\begin{equation}\label{ga4}
\frac{\partial^2 u}{\partial t^2}-\Delta u\geq C w^{-(\beta+2)/2}
\quad\mbox{ in }\mathbb R^{N+1}_+, 
\end{equation}
for some constant $C>0$.
To proceed further, we need the following result.
\begin{lemma}\label{inte}
Let $\alpha\in (0, N)$, $\sigma>N-\alpha$ and $1\leq \lambda\leq 2$. Then, there exists a constant $C=C(N,\alpha, \sigma)>0$ (note that $C$ is independent of $\lambda$) such that 
$$
\int_{\mathbb R^N} \frac{dy}{|x-y|^{\alpha}(\lambda+|y|^2)^{\sigma/2}} 
\leq
C\left\{
\begin{aligned}
&(\lambda+|x|^2)^{(N-\alpha-\sigma)/2}&&\quad\mbox{ if }\sigma<N,\\
&(\lambda+|x|^2)^{-\alpha/2} &&\quad\mbox{ if }\sigma>N,
\end{aligned}
\right.
\qquad\mbox{ for all }|x|\geq 1.
$$
\end{lemma}
For $\lambda=0$, similar estimates are available in  \cite[Lemma A.1]{MV2013} as well as in  \cite{GMM2021,GT2015,GT2016book}.

\begin{proof} Let $|x|\geq 1$.  We split the integral into 
\begin{equation*}
\int_{\mathbb R^N} \frac{dy}{|x-y|^{\alpha}(\lambda+|y|^2)^{\sigma/2}}   =\left\{\, \int\limits_{|y|\geq 2|x|}  +
\int\limits_{\frac12 |x| \leq |y| \leq 2|x|} + \int\limits_{|y|\leq  |x|/2} \right\} \frac{dy}{|x-y|^{\alpha}(\lambda+|y|^2)^{\sigma/2}}.
\end{equation*}
If $|y|\geq 2|x|$, then $|x-y| \geq |y|-|x| \geq |y|/2$ and we have 
\begin{equation*}
\begin{aligned}
\int\limits_{|y|\geq 2|x|} \frac{dy}{|x-y|^{\alpha}(\lambda+|y|^2)^{\sigma/2}}  
&\leq C \int\limits_{|y|\geq 2|x|} \frac{dy}{|y|^{\alpha+\sigma}}
\leq C|x|^{N-\alpha -\sigma}\\[0.1in]
&\leq C\Big(\frac{\lambda+|x|}{3}\Big)^{N-\alpha -\sigma}\\
&\leq C\big(\lambda+|x|\big)^{N-\alpha -\sigma}\\
&\leq C (\lambda+|x|^2)^{(N-\alpha-\sigma)/2},
\end{aligned}
\end{equation*}
where we have used that 
\begin{equation} \label{ineq_lambda}|x|\ge \frac{\lambda +|x|}3 \quad\text{if}\quad  |x|\ge1 \quad\text{and} \quad  \lambda\le2.\end{equation}
Next we have
\begin{align*}
\int\limits_{\frac12 |x| \leq |y| \leq 2|x|} \frac{dy}{|x-y|^{\alpha}(\lambda+|y|^2)^{\sigma/2}}  
&\leq C(\lambda+|x|^2)^{-\sigma/2}  \int\limits_{\frac12 |x| \leq |y| \leq 2|x|}
\frac{dy}{|x-y|^{\alpha}}\\
&\leq C(\lambda+|x|^2)^{-\sigma/2}  \int\limits_{ |y-x| \leq 3|x|}
\frac{dy}{|x-y|^{\alpha}}\\
&\leq C(\lambda+|x|^2)^{-\sigma/2}   |x|^{N-\alpha}\\
&\leq C(\lambda+|x|^2)^{(N-\alpha-\sigma)/2},
\end{align*}
where we have used that $|y|\le 2|x|$  implies $|y-x|\le 3|x|$.

Finally, if $|y|\leq |x|/2$ then $|x-y| \geq |x|-|y| \geq |x|/2$. We have
\begin{equation}\label{ga5}
\int\limits_{|y|\leq  |x|/2} \frac{dy}{|x-y|^{\alpha}(\lambda+|y|^2)^{\sigma/2}}  \leq
C|x|^{-\alpha}  \int\limits_{|y|\leq  |x|/2} \frac{dy}{(\lambda+|y|^2)^{\sigma/2}}.  
\end{equation}
If $\sigma<N$ then 
$$
\int\limits_{|y|\leq  |x|/2} \frac{dy}{(\lambda+|y|^2)^{\sigma/2}}\leq \int\limits_{|y|\leq  |x|/2} \frac{dy}{|y|^ \sigma}=C|x|^{N-\sigma},
$$
so that by  \eqref{ineq_lambda} and being $N-\alpha-\sigma<0$
$$
\int\limits_{|y|\leq  |x|/2} \frac{dy}{|x-y|^{\alpha}(\lambda+|y|^2)^{\sigma/2}} \leq C|x|^{N-\alpha-\sigma}
\leq C\Big(\frac{\lambda+|x|}{3}\Big)^{N-\alpha -\sigma}
\leq C (\lambda+|x|^2)^{(N-\alpha-\sigma)/2}.
$$
If $\sigma>N$ then 
$$
\int\limits_{|y|\leq  |x|/2} \frac{dy}{(\lambda+|y|^2)^{\sigma/2}}\leq \int\limits_{\mathbb R^N} \frac{dy}{(1+|y|^2)^{\sigma/2}}<C<\infty,
$$
and from \eqref{ga5} and  \eqref{ineq_lambda} one has
$$
\int\limits_{|y|\leq  |x|/2} \frac{dy}{|x-y|^{\alpha}(\lambda+|y|^2)^{\sigma/2}}  \leq
C|x|^{-\alpha} \leq C\Big(\frac{\lambda+|x|}{3}\Big)^{-\alpha}
\leq C (\lambda+|x|^2)^{-\alpha/2},
$$
which concludes our proof. 
\end{proof}
Let us return to the proof of Corollary \ref{cor2} and observe that
$$
\big(|x|^{-\alpha}*u^p\big)u^q\leq (1+|x|^2)^{-\beta q/2}\int\limits_{\mathbb R^N}\frac{dy}{|x-y|^{\alpha}(1+|y|^2)^{\beta p/2}}\quad\mbox{ for all }(x,t)\in \mathbb R^{N+1}_+,
$$
since $w(x,t)\ge 1+|x|^2$.
Since from \eqref{ga1} we have $\beta p>0$ and $\alpha<N$, the above integral is finite and thus $
\big(|x|^{-\alpha}*u^p\big)u^q$ is uniformly bounded from above for $(x,t)\in B_1\times [0,\infty)$. {In addition, since $1\le w\le3$ in $(x,t)\in B_1\times [0,\infty)$, then
by \eqref{ga4}}  $\frac{\partial^2 u}{\partial t^2}-\Delta u$ and $
\big(|x|^{-\alpha}*u^p\big)u^q$ are positive, continuous and uniformly bounded functions from below and from above respectively, on $(x,t)\in B_1\times [0,\infty)$.  We may thus find $C_1>0$ such that
\begin{equation}\label{ga6}
\frac{\partial^2 u}{\partial t^2}-\Delta u\geq C_1 \big(|x|^{-\alpha}*u^p\big)u^q\quad\mbox{ in } B_1\times [0,\infty).
\end{equation}
On the other hand, by Lemma \ref{inte} for $\lambda=1+(1+t)^{-a}$ and $\sigma=\beta p$, where $\sigma>N-\alpha$ by \eqref{ga1},  we have
$$
\big(|x|^{-\alpha}*u^p\big)u^q \leq 
C\left\{
\begin{aligned}
&w^{\frac{N-\alpha-\beta(p+q)}{2}}&&\quad\mbox{ if }\sigma=\beta p<N,\\
&w^{-\frac{\beta q+\alpha}{2}} &&\quad\mbox{ if }\sigma=\beta p>N,
\end{aligned}
\right.
\qquad\mbox{ for all }(x,t)\in (\mathbb R\setminus B_1)\times [0,\infty).
$$
Using the above estimate together with \eqref{ga1} and \eqref{ga4} we find
\begin{equation}\label{ga7}
\frac{\partial^2 u}{\partial t^2}-\Delta u\geq C w^{-(\beta+2)/2}\geq C w^{\max\{\frac{N-\alpha-\beta(p+q)}{2}, -\frac{\beta q+\alpha}{2} \}}\geq C_2 \big(|x|^{-\alpha}*u^p\big)u^q, 
\end{equation}
for all $(x,t)\in (\mathbb R\setminus B_1)\times [0,\infty)$. 
Letting now $M=\big(\max\{C_1,C_2\}\big)^{1/ (p+q )-1}$, it follows from \eqref{ga6} and \eqref{ga7} that 
$U=Mu$ is a $C^\infty(\mathbb R^{N+1}_+)$ solution of \eqref{mainh} such that 
$$
\frac{\partial U}{\partial t}(x,0)=\frac{a\beta M}{2}\big(2+|x|^2\big)^{-(\beta+2)/2}>0.
$$

\section{Proof of Theorem \ref{thmainsystem}}

Let $(u,v)$ be a  weak solution of \eqref{mains} which satisfies \eqref{int_uv}. Let $\varphi$ be the test function given by \eqref{test_function} with $\gamma=2m/k$. Arguing as in \eqref{weak_0} we have
\begin{equation} \label{weaksystem}
\left\{
\begin{aligned}
\int_{\mathbb R^N}u_{k-1}(x)\varphi(x,0)+\int_{\mathbb R^{N+1}_+} (K\ast |v|^p)|v|^q  \varphi \, dxdt & \le 
\int_{\mathbb R^{N+1}_+}u\left(  \Big|\frac{\partial^k \varphi}{\partial t^k}\Big|+|\Delta^m\varphi| \right) dx dt,  \\
\int_{\mathbb R^N}v_{k-1}(x)\varphi(x,0)+\int_{\mathbb R^{N+1}_+} (L\ast |u|^n)|u|^{ s}  \varphi \, dxdt & \le 
\int_{\mathbb R^{N+1}_+}v \left(\Big|\frac{\partial^k \varphi}{\partial t^k}\Big|dxdt+ |\Delta^m\varphi|\right) dx dt.  \\
\end{aligned}
\right.
\end{equation}
Setting
$$
\begin{aligned}
I(t)&=\int_{\mathbb R^N}|u(x,t)|^{\frac{n+s}2} \varphi(x,t) dx,\\
J(t)&=\int_{\mathbb R^N}|v(x,t)|^{\frac{p+q}2} \varphi(x,t) dx,
\end{aligned}$$
we have supp $I\subset [0, 2R^\gamma)$ and  supp $J\subset [0, 2R^\gamma)$.
With the same approach as in Lemma \ref{lk}, for $R>R_0$  we have
\begin{equation}\label{IJ}
\left\{
\begin{aligned}
\int_{\mathbb R^N}(K\ast |v|^p)|v(x,t)|^q  \varphi(x,t) dx\ge K(4R) J^2(t),\\
\int_{\mathbb R^N}(L\ast |u|^n)|u(x,t)|^s  \varphi(x,t) dx\ge L(4R) I^2(t),
\end{aligned}
\right.
\qquad\mbox{ for all }\;t\geq 0.
\end{equation}
Thus, \eqref{weaksystem}, \eqref{sign_u1_rho}  and \eqref{int_uv} yield
\begin{equation} \label{weaksystem0}
\left\{
\begin{aligned}
K(4R) \int_0^{2R^{\gamma}}J^2(t) \, dt & \le 
\int_{\mathbb R^{N+1}_+} |u| \left(\Big|\frac{\partial^k \varphi}{\partial t^k}\Big|+ |\Delta^m\varphi|\right) dx dt,  \\
L(4R) \int_0^{2R^{\gamma}} I^2(t) \, dt & \le 
\int_{\mathbb R^{N+1}_+} |v| \left( \Big|\frac{\partial^k \varphi}{\partial t^k}\Big| + |\Delta^m\varphi|\right) dx dt,  \\
\end{aligned}
\right.
\end{equation}
for $R>0$ large.
We next employ the estimates \eqref{rho_tt<} and \eqref{Delta_m_rho<} with $\gamma=2m/k$. We find 
$$
\begin{aligned}
 \int_{\mathbb R^{N+1}_+}&   |u| \left(\Big|\frac{\partial^k \varphi}{\partial t^k}\Big|  +|\Delta^m \varphi|\right) dxdt\\
&\le C R^{-2m+\big(N+\frac{2m}{k}\big)\frac{n+s-2}{n+s} }\left[ \biggl(\int_{\text{supp}(\rho_{tt})} |u|^{\frac{n+s}{2}}\varphi \,dxdt\biggr)^{\frac{2}{n+s}} + \biggl(\int_{\text{supp}(\Delta^m\rho)} |u|^\frac{n+s}{2} \varphi \,dxdt\biggr)^{\frac{2}{n+s}} \right].
\end{aligned}
$$
Since supp$(\frac{\partial^k \varphi}{\partial t^k})\subset [R^\gamma, 2R^\gamma]\times \mathbb R^N$ and supp$(\Delta^m \varphi)\subset [0, 2R^\gamma]\times (B_{2R}\setminus B_R)$, applying H\"older's inequality we derive

\begin{equation} \label{weaksystem1}
\begin{aligned}
\int\limits_{\mathbb R^{N+1}_+}   |u| & \left(\Big|\frac{\partial^k \varphi}{\partial t^k}\Big|  +|\Delta^m \varphi|\right) dxdt\le C R^{N-2m\big(1-\frac{1}{k}\big)-\frac{2N+2m/k}{n+s} } \times \\
&\times \left[ \biggl(\int_{R^\gamma}^{2R^\gamma}\Big(\int_{\mathbb R^N} |u|^{\frac{n+s}{2}}\varphi \,dx\Big)^2 dt\biggr)^{ \frac{1}{n+s} } + \biggl(\int_0^{2R^\gamma}\Big(\int_{B_{2R}\setminus B_R}
  |u|^\frac{n+s}{2} \varphi \,dx\Big)^2 dt\biggr)^{ \frac{1}{n+s}} \right].
\end{aligned}
\end{equation}
In particular,
\begin{equation} \label{weaksystem01}
\begin{aligned}
 \int_{\mathbb R^{N+1}_+} &  | u| \left(\Big|\frac{\partial^k \varphi}{\partial t^k}\Big|  +|\Delta^m \varphi|\right) dxdt
&\le C R^{N-2m\big(1-\frac{1}{k}\big)-\frac{2N+2m/k}{n+s} } \Big(\int_0^{2R^\gamma} I^2(t)\, dt\Big)^{\frac{1}{n+s}}.
\end{aligned}
\end{equation}
Similarly
\begin{equation} \label{weaksystem2}
\begin{aligned}
\int\limits_{\mathbb R^{N+1}_+}  |v|  & \left(\Big|\frac{\partial^k \varphi}{\partial t^k}\Big|  +|\Delta^m \varphi|\right) dxdt\le C R^{N-2m\big(1-\frac{1}{k}\big)-\frac{2N+2m/k}{p+q} }  \times \\
&\times  \left[ \biggl(\int_{R^\gamma}^{2R^\gamma}\Big(\int_{\mathbb R^N} |v|^{\frac{p+q}{2}}\varphi \,dx\Big)^2 dt\biggr)^{\frac{1}{p+q} } + \biggl(\int_0^{2R^m}\Big(\int_{B_{2R}\setminus B_R}  |v|^\frac{p+q}{2} \varphi \,dx\Big)^2 dt\biggr)^{ \frac{1}{p+q} } \right],
\end{aligned}
\end{equation}
and in particular, 
\begin{equation} \label{weaksystem02}
\begin{aligned}
\int_{\mathbb R^{N+1}_+} & |v| \left(\Big|\frac{\partial^k \varphi}{\partial t^k}\Big|  +|\Delta^m \varphi|\right) dxdt
&\le C R^{N-2m\big(1-\frac{1}{k}\big)-\frac{2N+2m/k}{p+q} }    \Big(\int_0^{2R^\gamma} J^2(t)\, dt\Big)^{\frac{1}{p+q} }.
\end{aligned}
\end{equation}
From \eqref{weaksystem0}, \eqref{weaksystem01} and \eqref{weaksystem02} we  derive
\begin{equation} \label{weaksystem3}
\left\{
\begin{aligned}
K(4R) \int_0^{2R^\gamma} J^2(t) \, dt & \le C 
R^{N-2m\big(1-\frac{1}{k}\big)-\frac{2N+2m/k}{n+s} }  \Big(\int_0^{2R^\gamma} I^2(t)\, dt\Big)^{ \frac{1}{n+s}},  \\
L(4R) \int_0^{2R^{ \gamma}} I^2(t) \, dt & \le 
C R^{N-2m\big(1-\frac{1}{k}\big)-\frac{2N+2m/k}{p+q} }  \Big(\int_0^{2R^\gamma} J^2(t)\, dt\Big)^{ \frac{1}{p+q} }. \\
\end{aligned}
\right.
\end{equation}
Assume now that \eqref{Ksup1} holds and let $\{R_j\}$ be an increasing sequence of positive real numbers such that $R_j>2R_{j-1}$ for all $j\geq 2$ and 
\begin{equation}\label{eqj1}
K(R_j) L(R_j)^{ \frac{1}{n+s} }   R_{j}^{\frac{2N+2m/k}{(n+s)(p+q)}+\frac{N+2m}{n+s}-N+2m\big(1-\frac{1}{k}\big) }\longrightarrow \ell\in (0, \infty]\quad\mbox{ as }j\to \infty.
\end{equation}
Assume first $\ell=\infty$. Then, using the second estimate of \eqref{weaksystem3} into the first we find
\begin{equation}\label{eqj2}
\Big(\int_0^{2R_j^\gamma} J^2(t) \, dt \Big)^{ 1-\frac{1}{(n+s)(p+q)} } \leq \frac{C}{K(R_j) L(R_j)^{\frac{1}{n+s}}   R_{ j}^{\frac{2N+2m/k}{(n+s)(p+q)}+\frac{N+2m}{n+s}-N+2m\big(1-\frac{1}{k}\big)}}\longrightarrow 0 
\end{equation}
as $j\to \infty$. This implies that $J(t)\equiv 0$ so that $v\equiv 0$. From the second estimate of \eqref{weaksystem3} it follows now that $I(t)\equiv 0$, hence $u\equiv 0$.

Assume now $\ell\in (0, \infty)$. From \eqref{eqj2} it follows that $J^2\in L^1(0, \infty)$ and arguing as in \eqref{eqeq}, this further implies
\begin{equation}\label{eqj3}
\int_{R_j^\gamma}^{2R_j^{\gamma}} \Big(\int_{\mathbb R^N} |v(x,t)|^{\frac{p+q}{2}}\varphi(x,t) dx\Big)^2 dt\to 0\,,\;
\int_{0}^{2R_j^{\gamma}} \Big(\int_{B_{2R_j}\setminus B_{R_j} } |v(x,t)|^{\frac{p+q}{2}}\varphi(x,t) dx\Big)^2 dt\to 0, 
\end{equation}
as $j\to \infty$. 
Using \eqref{eqj3} into \eqref{weaksystem2} (with $R_j$ instead of $R$) it follows that  
$$
\int_{\mathbb R^{N+1}_+} |v| \left( \Big|\frac{\partial^k \varphi}{\partial t^k}\Big|  +|\Delta^m \varphi|\right) dxdt
\le C R^{N-2m\big(1-\frac{1}{k}\big)-\frac{2N+2m/k}{p+q} } o(1)\quad\mbox{ as }j\to \infty.
$$
Using this fact in \eqref{weaksystem0} along with \eqref{weaksystem01} we deduce
\begin{equation} \label{eqj4}
\left\{
\begin{aligned}
K(4R) \int_0^{2R^\gamma} J^2(t) \, dt & \le C 
R^{N-2m\big(1-\frac{1}{k}\big)-\frac{2N+2m/k}{n+s} }  \Big(\int_0^{2R^\gamma} I^2(t)\, dt\Big)^{ \frac{1}{n+s} }  \\
L(4R) \int_0^{2R^\gamma} I^2(t) \, dt & \le 
C R^{N-2m\big(1-\frac{1}{k}\big)-\frac{2N+2m/k}{p+q} }  o(1) \\
\end{aligned}
\right.
\quad \mbox{ as }j\to \infty.
\end{equation}
Using the second estimate of \eqref{eqj4} into the first one obtains
$$
\Big(\int_0^{2R_j^\gamma} J^2(t) \, dt \Big)^{ 1-\frac{1}{(n+s)(p+q)} }\leq \frac{C}{K(R_j) L(R_j)^{\frac{2}{n+s}}  R_{ j}^{ \frac{2N+2m/k}{(n+s)(p+q)}+\frac{N+2m}{n+s}-N+2m\big(1-\frac{1}{k}\big)}}\, o(1) \longrightarrow 0, 
$$
as $j\to \infty$. As before we deduce again $u\equiv v\equiv 0$.

\section*{Acknowledgments}

We thank the Referees for the valuable comments which helped us to prepared this improved version of our results.

\noindent RF is a member of the {\it Gruppo Nazionale per l’Analisi Matematica, la Probabilit\`a e le loro Applicazioni (GNAMPA) of the Istituto Nazionale di Alta Matematica} (INdAM) and she was partly supported by {\it Fondo Ricerca di Base di Ateneo Esercizio 2017–19} of the University of Perugia, Italy, titled {\it Problemi con non linearit\`a dipendenti dal gradiente}.

\end{document}